\documentclass[a4paper,oneside,portrait,11pt]{AmsArt}

\usepackage[margin=1in]{geometry}

\usepackage[ps,all,arc,rotate]{xy}
 
\usepackage{amsfonts,amscd,amsthm,amsgen,amsmath,amssymb}
\usepackage[all]{xy}
\usepackage[vcentermath]{youngtab}
\usepackage {graphicx}
\usepackage{epstopdf}
\usepackage{amsfonts}
\usepackage{amsthm}
\usepackage{amsmath}
\usepackage{amsfonts}
\usepackage{latexsym}
\usepackage{amssymb}
\usepackage{epsfig,color}
\usepackage{graphicx, amssymb}
 \usepackage{color}
 \usepackage{epsfig,color}
\usepackage{graphicx}
\usepackage{caption}
\usepackage{subcaption}
\usepackage{wrapfig}
\usepackage {graphicx}
\usepackage{amsfonts}
\usepackage{amsthm}
\usepackage{amsmath}
\usepackage{amsfonts}
\usepackage{latexsym}
\usepackage{amssymb}
\usepackage{epsfig,color}
\usepackage{graphicx, amssymb}
\usepackage{epstopdf}

\usepackage{hyperref}
\hypersetup{
  bookmarks=true,
}

\newcommand{\M}{\mathcal{M}_{reg}}

\newtheorem{definition}[]{Definition} 
\newtheorem{theorem}[definition]{Theorem}
\newtheorem{corollary}[definition]{Corollary}
\newtheorem{proposition}[definition]{Proposition}
\newtheorem{remark}[definition]{Remark}
\newtheorem{lemma}[definition]{Lemma}
\theoremstyle{definition}

\def\CF{{\mathcal F}}

\def\CL{{\mathcal L}}
\def\CM{{\mathcal M}}

\def\CO{{\mathcal O}}

\def\CV{{\mathcal V}}

\newcommand{\C}{\mathbb{C}}
\newcommand{\Z}{\mathbb{Z}}
\newcommand{\R}{\mathbb{R}}

\newcommand{\Sp}{{ Sp}}

\title{A geometric approach to orthogonal Higgs bundles }
 \author{ Laura P. Schaposnik  }

\date{Department of Mathematics, University of Illinois, Chicago, 60647, USA. 
\\ \indent ~~schaposnik$at$uic.edu}

\begin{document}

\baselineskip=1.1\baselineskip

\begin{abstract}  We give a geometric characterisation of the topological invariants associated to $SO(m,m+1)$-Higgs bundles  through KO-theory and the Langlands correspondence between orthogonal and symplectic Hitchin systems. By defining the split orthogonal spectral data, we obtain a natural grading of the moduli space of  $SO(m,m+1)$-Higgs bundles.
  \end{abstract}
\maketitle

\section{Introduction}

  Higgs bundles were first studied by Nigel Hitchin in 1987, and appeared as solutions of Yang-Mills self-duality equations on a Riemann surface \cite{N1}. Classically, a {\it Higgs bundle} on a compact Riemann surface $\Sigma$ of genus $g\geq 2$ is a pair $(E,\Phi)$ where 
 $E$ is a holomorphic vector bundle  on $\Sigma$, 
and   the {\it  Higgs field} $\Phi: E\rightarrow E\otimes K$ is a holomorphic 1-form,
 for $K:=T^*\Sigma$.  
Higgs bundles can also be defined for complex semisimple groups $G_{\C}$ and their real forms, and through stability conditions one can construct their moduli spaces $\mathcal{M}_{G_\C}$ (e.g. see \cite{N2}).  

 A natural way of studying the moduli space of Higgs bundles   is through the {\it Hitchin fibration}, 
  sending the class of a Higgs bundle $(E,\Phi)$ to the coefficients of the characteristic polynomial $\det(\eta I -\Phi )$. The generic fibre  is an abelian variety   which can be seen through line bundles on an algebraic curve $S$, the {\it spectral curve} associated to the Higgs field as introduced in \cite{N2}. The \textit{spectral data}  is then given by the line bundle on $S$ satisfying certain conditions. In the case of classical Higgs bundles, the smooth fibres are Jacobian varieties of $S$.  The Hitchin fibration was   defined for classical complex Lie groups in \cite[Section 5]{N2}, and following \cite[Section 7]{N5} one may consider Higgs bundles with real structure group $G$ as fixed point sets in the moduli space of Higgs bundles for the complexified group $G_\C$, therefore obtaining $G$-Higgs bundles as real points inside the Hitchin fibration (e.g. see \cite{N5, gothen,thesis} and references therein).

We dedicate this short note to the study of the geometry of the moduli space of   $SO(m,m+1)$-Higgs bundles inside $\CM_{SO(2m+1,\C)}$. In particular, whilst in this non hermitian symmetric case there is no Toledo invariant, one can consider the Langlands dual set up of $Sp(2m,\C)$-Higgs bundles $(E=W\oplus W^*,\Phi')$  to understand the role of the symplectic Toledo invariant from the orthogonal perspective, as well as to construct the spectral data.  
By considering real Higgs bundles as fixed points of an involution (e.g. see \cite[Section 3.3.1]{thesis}), we see the moduli space of $SO(m,m+1)$-Higgs bundles $(V=V_+\oplus V_-, \Phi)$ inside the $SO(2m+1,\C)$-Hitchin fibration. 

The characteristic polynomials of   $SO(m,m+1)$ and $Sp(2m,\R)$-Higgs fields define a spectral curve $\pi:S\rightarrow \Sigma$ in the total space of the canonical bundle $K$ whose equation is 
\begin{eqnarray}\eta^{2m}+a_1\eta^{2m-2}+\ldots+a_{m-1}\eta^2+a_m=0,\label{yo}\end{eqnarray}
for $\eta$ the tautological section of $\pi^*K$ and $a_i\in H^0(\Sigma, K^{2i})$ . This is a $2m$ covering of the Riemann surface, generically smooth and ramified over $4m(g-1)$ points, the zeros of $a_m$. 
In order to understand the topological invariants associated to  $SO(m,m+1)$-Higgs bundles, one has to consider a subdivisor $D$ of the ramification divisor, over which a natural involution $\sigma:\eta\mapsto -\eta$ acts as $-1$, and whose degree we denote by $M$ following the notation of \cite{umm}. The value of $M$ is closely related to the {\it Toledo invariant} (see \cite[Section 6]{classes}), and in particular one can deduce the following:
%
  \vspace{0.05 in}

\noindent {\bf Theorem \ref{comp-sp}.} {\it Each even  invariant $0<M\leq 4m(g-1)$  labels   a component of the moduli space of $Sp(2m,\R)$-Higgs bundles which intersects the nonsingular fibres of the Hitchin fibration for $Sp(2m,\C)$-Higgs bundles given by a fibration of a $\Z_2$-vector space over the total space of a vector bundle on the symmetric
product $S^M\Sigma$.}
 \vspace{0.05 in}

In the case of  orthogonal Higgs bundles, one has the following:
 \vspace{0.05 in}

\noindent {\bf  Proposition \ref{fibreSO}.} {\it 
 The intersection of the moduli space $\CM_{SO(m,m+1)}$ with the regular fibres of the $SO(2m+1,\C)$-Hitchin fibration is given by two copies of }
 \begin{eqnarray}
\{~L\in {\rm Prym}(S,S/\sigma)~:~L^2\cong \CO~\}~/~ H^1(S/\sigma, \Z_2).\nonumber
\end{eqnarray}

Each of the two copies corresponds to whether the orthogonal bundle lifts to a spin bundle or not. Moreover, there is a decomposition of the torsion 2 points in the Prym variety ${\rm Prym}(S,S/\sigma)[2] \cong H^1(S/\sigma, \Z_2)\oplus \Z_2([a_m])^{ev}/b_0,$ 
where $\Z_2([a_m])^{ev}$ denotes subdivisors of the divisor $[a_m]$ with even number of +1, and $b_0:=(1,\ldots, 1)$.
Thus   the spectral data of an $SO(m,m+1)$-Higgs bundle is given, up to equivalence,  by 
\begin{itemize}
\item a line bundle $\CF\in H^1(S/\sigma, \Z_2)$, and
\item a divisor $D\in \Z_2([a_m])^{ev}/b_0$ of degree $M$. 
\end{itemize}

 Since $SO(m,m+1)$ retracts onto $S(O(m)\times O(m+1))$,  an $SO(m,m+1)$-Higgs bundle $(V_{+}\oplus V_{-}, \Phi)$ carries three topological invariants: the Stiefel-Whitney classes $\omega_1(V_+)=\det(V_+)$, and $\omega_2(V_\pm)$. Through a $K$-theoretic approach following the methods of \cite{slice,classes}, in Section \ref{section-so} we can further classify these invariants in terms of their spectral data: 
 \vspace{0.05 in}

\noindent {\bf Theorem \ref{teo2}.} {\it  The  Stiefel-Whitney classes  of an $SO(m,m+1)$-Higgs bundle $(V=V_-\oplus V_+, \Phi)$ with spectral data $(S/\sigma, \CF, D)$   are given by}
\begin{eqnarray}
 \omega_1(V_+)&=&{\rm Nm}(\CF)  \in H^1(\Sigma, \Z_2);\\
\omega_2(V_+)&=&\varphi_{  S
/\sigma}(\CF)+\varphi_\Sigma({\rm Nm}(\CF)) \in \Z_2;\\
\omega_2(V_-)&=&\left\{\begin{array} {ccc}
\varphi_{  S
/\sigma}(\CF)+\varphi_\Sigma({\rm Nm}(\CF)) &{\rm if} &\omega_2(V)=0\\
\varphi_{  S
/\sigma}(\CF)+\varphi_\Sigma({\rm Nm}(\CF)) +1&{\rm if} &\omega_2(V)=1
\end{array}\right.
 \end{eqnarray}
 {\it for $\varphi_\Sigma$ and $\varphi_{ S/\sigma}$ the analytic mod 2 indices of the curves, and ${\rm Nm}(\CF)$ the Norm  on $\Sigma$. } 
  \vspace{0.05 in}

By considering $\CF\otimes K_{\bar S}$ as a new spin structure,   one can see $\omega_{2}(V_{\pm})$ purely in terms of spin structures in Corollary \ref{corro}. 
Moreover, by analysing spectral data through the induced 2-fold cover $\rho : S\rightarrow \bar S:=S/\sigma$, and recalling that  the orthogonal vector bundle  $V_+\oplus V_-$ is recovered as an extension   defined through the divisor $D$ (see \cite[Section 4.2]{N3}), 
 one obtains the number of points in each of the regular fibres of the Hitchin fibration for a fixed invariant $M$. \vspace{0.05 in}

\noindent {\bf  Proposition \ref{numbers}.} {\it The number of points in a regular fibre    of the $SO(2m+1,\C)$ Hitchin fibration corresponding to $SO(m,m+1)$-Higgs bundles with even invariant $M$ is}
\begin{small}$\left(\begin{array}{c} 4m(g-1)\\M
\end{array}\right).\nonumber
$ \end{small}

\vspace{0.05 in}

By considering the parametrisation of the moduli space through  spectral data, we obtain a natural grading of the moduli space of $SO(m,m+1)$-Higgs bundles leading to a description of Zariski dense open sets in each  component:
 \vspace{0.05 in}

\noindent {\bf Theorem \ref{teo1}.} {\it Each fixed even invariant $0<M\leq 4m(g-1)$ labels a component of the moduli space of   $SO(m,m+1)$-Higgs bundles whose intersection with the regular fibres of the Hitchin fibration is a  covering of a vector space over the symmetric product $S^M\Sigma$.}

 \vspace{0.05 in}

It is important to note that these components will possibly (and often do so) meet over the discriminant locus of the Hitchin fibration, and thus one needs to do further analysis to understand the connectivity of the moduli space. An example of how to see the intersection of the components through the monodromy of the associated {\it Gauss-Manin connection} for $SO(2,3)$-Higgs bundles is discussed in   \cite[Section 6.3]{mono}. 

A geometric description of the above covering is given in Section \ref{geometrytori}, recovering some of the results appearing in  \cite[Section 6.4]{brian}.
Moreover, in the so-called {\it maximal Toledo invariant case} on the symplectic side, which corresponds to $M=0$, one can deduce  that when $m$ is odd the intersection of $\CM_{SO(m,m+1)}$ with the smooth fibres of the $SO(2m+1,\C)$ Hitchin fibration is given by $2^{2g}$ copies of ${\rm Prym}(\bar S,\Sigma)$ over a vector space.

The moduli space of $SO(m,m+1)$-Higgs bundles considered in this paper is an example of what is known as $(B,A,A)$-brane in the moduli space $\CM_{SO(2m+1,\C)}$ of complex Higgs bundles. As such,  these branes have dual $(B,B,B)$-branes in the dual moduli space $\CM_{Sp(2m,\C)}$  (see \cite[Section 12]{LPS_Kap}). In \cite[Section 7]{slice} it was conjectured what the support of this dual brane should be,  the whole moduli space $\CM_{Sp(2m,\C)}$ of symplectic Higgs bundles.  We conclude this   note with some further comments on this duality  in Section \ref{last}, as well as on the relation between the Hitchin components in both split symplectic and orthogonal  $(B,A,A)$-branes in Langlands dual groups, and some implications of the geometric description of the spectral data given in this paper. 

 \subsubsection*{Acknowledgements} This research was inspired by fruitful conversations with David Baraglia, \linebreak Steve Bradlow and Nigel Hitchin, and the author is also thankful for discussions with Ben Davidson and Alan Thompson.   The paper  was written with partial support of the U.S. National Science Foundation grants DMS 1107452, 1107263, 1107367:  the GEAR Network for short research visits. The work of the author is also supported by NSF grant DMS-1509693.

\section{The Hitchin fibration}\label{fibration}
Recall from \cite{N2} that an $Sp(2m,\mathbb{C})$-Higgs bundle is a pair $(E,\Phi')$ for $E$ a rank $2m$ vector bundle  with a symplectic form $\omega(~,~)$, and the Higgs field $\Phi'\in H^{0}(\Sigma, {\rm End}(E)\otimes K)$ satisfying
$\omega(\Phi' v,w)=-\omega(v,\Phi' w).$ 
Similarly, an $SO(2m+1,\mathbb{C})$-Higgs bundle is a pair $(V,\Phi)$ for $V$ a holomorphic vector bundle of rank $2m+1$ with  a non-degenerate symmetric bilinear form $(v,w)$, and  $\Phi$ a Higgs field  in $H^{0}(\Sigma,{\rm End}_{0}(V)\otimes K)$ which satisfies
$(\Phi v,w)=-(v,\Phi w).$

The spectral curves defined by $SO(2m+1,\C)$-Higgs bundles  and $Sp(2m,\C)$-Higgs bundles  have similar equations (e.g. see \cite[Section 3-4]{N3}), and are given by a $2m$-fold cover $\pi:S\rightarrow \Sigma$ in the total space of $K$ whose equation is
$\eta^{2m}+a_{1}\eta^{2m-2}+\ldots +a_{m-1}\eta^{2}+a_{m}=0$ as in \eqref{yo}. 
 The  curve $S$ has an involution $\sigma$ which acts as $\sigma(\eta)=-\eta$ and thus we may consider the quotient curve $ \overline{S}:=S/\sigma$ in the total space of $K^{2}$,  for which $\rho: S\rightarrow \bar S$ is a double cover: \begin{eqnarray}
\xymatrix{ S\ar[rd]_{2m:1}^{\pi}\ar[rr]^{2:1}_{\rho}&&\bar S\ar[ld]_{\bar \pi}^{m:1}\\
&\Sigma &}
\end{eqnarray} 
 
 The covers $S$ and $\bar S$ have, respectively,  genus $ g_{S}= 1+4m^{2}(g-1),$  and $g_{\bar S}=(2m^2-m)(g-1)+1$. Moreover, by the adjunction formula, their canonical bundles can be written, respectively,   as
$K_S=\pi^*K^{2m}$ and
$K_{\bar S}=\bar \pi^*K^{2m-1}.
$ 
 As shown in \cite{N2}, the Hitchin fibration for both moduli spaces $\CM_{SO(2m+1,\C)}$ and $\CM_{Sp(2m,\C)}$ is given  over $\mathcal{A}=\bigoplus_{i=1}^{m}H^0(\Sigma, K^{2i}).$
From \cite[Section 3]{N3}, the generic fibres for $\mathcal{M}_{Sp(2m,\C)}$ are given by 
\begin{eqnarray}
{\rm Prym}(S,\bar S),\label{fibresp}
\end{eqnarray}
 and from \cite[Section 4]{N3}, the generic fibres for $\mathcal{M}_{SO(2m+1,\C)}$ are given by (two copies of)
 \begin{eqnarray}
 {\rm Prym}(S,\bar S)/\rho^*H^1(\bar S,\mathbb{Z}_2).\label{fibreso}
 \end{eqnarray}
 
 In what follows we shall study the components of the moduli space of Higgs bundles for split real forms by considering, from \cite[Theorem 4.12]{thesis}, points of order two in the generic fibres \eqref{fibresp} and \eqref{fibreso}.

\section{$Sp(2m,\R)$-Higgs bundles}\label{section-sp}

We shall now consider $Sp(2m,\R)$-Higgs bundles, which from \cite[Theorem 4.12]{thesis} can be seen in the generic fibres of the Hitchin fibration as points of order two in \eqref{fibresp}, and are given by $Sp(2m,\C)$-Higgs bundles which decompose as $(E=W\oplus W^*, \Phi')$.  
 Fixing a choice of $\Theta$ characteristic $L_0:=K^{1/2}$,  it is shown in \cite{N3} that the vector bundle $E$ is recovered as $\pi_*U$ for
  $U:=L  \otimes K^{(2m-1)/2}. $
   Note that the condition $L\in{\rm Prym}(S,\bar S)[2]:=\{ L\in {\rm Prym}(S,\bar S)~:~ L^2\cong \CO\}$ is equivalent to requiring
$U^2\cong K_S\pi^*K^*.$
Since points  in ${\rm Prym}(S,\bar S)[2]$ are given by line bundles $L$ on $S$ for which  $\sigma^*L\cong L^*\cong L ,$ following \cite[Theorem 3.5]{umm} they are classified by the action of the involution $\sigma$ on $L$ over its fixed point set (i.e., the ramification divisor of $S$). The involution $\sigma$ acts as $\pm 1$ over some subset of $M$ points of the ramification divisor $[a_m]$, and the number of Higgs bundles appearing in each fibre for a fixed invariant $M$ is described by Hitchin in \cite[Section 6]{classes}.  
  
  Higgs bundles for the real symplectic group have associated a topological invariant, the Toledo invariant (e.g. see \cite{gothen}), defined as
$ | \tau(W\oplus W^*, \Phi')|:=|c_1(W)|,\nonumber
$  and satisfy a Milnor-Wood type inequality $|c_1(W)|\leq m(g-1)$. 
  Moreover, from \cite[Section 6]{classes} the class can be expressed as
  \begin{eqnarray}
  w_1(W):=c_1(W)=-\frac{M}{2}+m(g-1),\nonumber\label{degree}
  \end{eqnarray}
 and its mod 2 value defines the invariant
$ c_1(W)~ ({\rm mod} ~2).\nonumber
$
 Note that since the invariant $M$ is even, within the moduli space of $Sp(2m,\R)$-Higgs bundles, the value of $c_1(W)~ ({\rm mod} ~2)$   differentiates components depending on the values of $M~({\rm mod~4})$. 
 \begin{remark} For $\Sp(2m,\C)$-Higgs bundles, the invariant $M$ appears as $n_-$ in \cite[Section 4]{N3}. In the case of $Sp(2m,\R)$-Higgs bundles, it is the invariant $l$ of \cite[Section 6]{classes}.
\end{remark}
 From \cite[Section 6]{classes}, the number of elements in ${\rm Prym}(S,\bar S)[2]$ corresponding to  $M$ is
\begin{eqnarray}
\left(\begin{array}{c} 4m(g-1)\\M
\end{array}\right)\times 2^{2g_{\bar S}}.\nonumber 
\end{eqnarray}
 In order to describe the geometry of the  components given by these Higgs bundles, recall from \cite[Proposition 4.15]{mono}, that a convenient splittings can be chosen for the short exact sequence
  \begin{eqnarray}
  0\rightarrow H^1(\bar S, \Z_2) \xrightarrow{\rho^*} {\rm Prym}(S,\bar S)[2]\rightarrow \Z_2([a_m])^{ev}/b_0\rightarrow 0,\nonumber
  \end{eqnarray}
  where $b_0$ is the divisor in $\Sigma$ which has $+1$ for all zeros of $a_m$. 
  Then, one can write 
  \begin{eqnarray}
 {\rm Prym}(S,\bar S)[2] \cong H^1(\bar S, \Z_2)\oplus \Z_2([a_m])^{ev}/b_0. \label{fibrespreal}\label{am}
  \end{eqnarray}
Hence, over each point in the base $\mathcal{A}$, one has  the set of divisors $D$ of degree $M$ over which the involution acts as $-1$ (as in \cite[Section 3]{umm}) given by a point in $ \Z_2([a_m])^{ev}/b_0$, together with the $\Z_2$ vector space $H^1(\bar S, \Z_2)$ of dimension 
$2g_{\bar S}= 2m(m-1)(g-1)+2.\nonumber
$

\begin{lemma}\label{lemmaH}
The space $H^1(\bar S, \Z_2)$ is given by
\begin{eqnarray}  
  {\rm Prym}(\bar S,\Sigma)[2]\oplus  \bar\pi^*H^1(\Sigma, \Z_2)~&{\rm for~}& m~{\rm odd},\label{seq-3} \\
 \bar\pi^*H^1(\Sigma, \Z_2)\oplus ({\rm Prym}(\bar S,\Sigma)[2]/  \bar\pi^*H^1(\Sigma, \Z_2))\oplus   \bar\pi^*H^1(\Sigma, \Z_2)~&{\rm for~}& m~{\rm even}.\label{seq-4}
 \end{eqnarray}
\end{lemma}
\begin{proof}
In order to understand the space $H^1(\bar S, \Z_2)$ one needs to take into account the parity of $m$, since $\bar S$ is an $m$-fold cover of the compact Riemann surface $\Sigma$. 
Considering the Norm map
\begin{eqnarray}
0\rightarrow {\rm Prym}(\bar S,\Sigma)[2]\rightarrow H^1(\bar S, \Z_2) \xrightarrow{Nm} H^1(\Sigma, \Z_2)\rightarrow 0, \label{seq-1}
\end{eqnarray}
note that for a line bundle $L$ on $\Sigma$ one has that $Nm(\bar\pi^*L)=mL$, and hence when $m$ is odd the pullback $\bar \pi^*$ gives a splitting of the short exact sequence \eqref{seq-1}. Therefore, in this case one has
\begin{eqnarray}
H^1(\bar S, \Z_2)\cong {\rm Prym}(\bar S,\Sigma)[2]\oplus \bar \pi^*H^1(\Sigma, \Z_2).\nonumber 
\end{eqnarray}

When $m$ is even, the image of $\bar \pi^*: H^1(\Sigma, \Z_2)\rightarrow H^1(\bar S, \Z_2)$ is contained  in ${\rm Prym}(\bar S,\Sigma)[2]$, thus giving a filtration
$ H^1(\Sigma, \Z_2)\subset {\rm Prym}(\bar S,\Sigma)[2]\subset  H^1(\bar S, \Z_2),$
which induces the splitting (via isomorphism theorems for $Nm$)
\begin{eqnarray}
H^1(\bar S, \Z_2)
 &\cong & \bar\pi^*H^1(\Sigma, \Z_2)\oplus ({\rm Prym}(\bar S,\Sigma)[2]/  \bar\pi^*H^1(\Sigma, \Z_2))\oplus   \bar\pi^*H^1(\Sigma, \Z_2),\nonumber
\end{eqnarray}
and thus the lemma follows.\end{proof}

From the above, one has the following description of Zariski open sets in   components of the moduli spaces of $Sp(2m,\R)$-Higgs bundles which intersect the smooth fibres of the $Sp(2m,\C)$ Hitchin fibration:

\begin{theorem}\label{comp-sp}
Each even  invariant $0<M\leq 4m(g-1)$  labels   one   component of the moduli space of $Sp(2m,\R)$-Higgs bundles which intersects the nonsingular fibres of the Hitchin fibration for $Sp(2m,\C)$-Higgs bundles. This component is given by a fibration of a $\Z_2$-vector space over the total space of a vector bundle on the symmetric
product $S^M\Sigma$.\end{theorem}

\begin{proof}
We shall follow the proof of \cite[Theorem 4.2]{umm}, taking into consideration the structure of the intersection of $\mathcal{M}_{Sp(2m,\R)}$ with the generic fibres of the Hitchin fibration given in \eqref{fibrespreal}. The invariant $M$ gives the degree of a divisor $D$ which corresponds to the choice of an element in $ \Z_2([a_m])^{ev}/d_0. $ Hence, it can be seen as a point in $S^M\Sigma$, and the choice of the differential $a_m$ is then given by a section in $H^0(\Sigma, K^{2m}(-D))$, or equivalently, a vector bundle over the symmetric product. Then, in order to define the spectral curve in the smooth loci of the Hitchin base one needs to make a choice of the remaining differentials, and thus a choice of a point in 
$\bigoplus_{i=1}^{2m-2}H^0(\Sigma, K^{2i}).\nonumber$
Finally, the Higgs bundles in the   component are obtained by the remaining data in the fibre, which is a point in the $\Z_2$ vector space $H^1(\bar S, \Z_2)$.\end{proof}

 From the above theorem one can see that when   $M=0$, the space is given by a $2^{2g_{\bar S}}$ cover of a vector space over a point, and it is the monodromy action which needs to be considered from this perspective in order to deduce connectivity of this cover - this study, for low rank, can be found in \cite{mono}. In particular, for $Sp(4,\R)$ it was shown by Gothen in \cite{gothen} that $M=0$ labels several connected components, and it is shown in \cite[Section 6.3]{mono} how these components appear as orbits of the monodromy action of the corresponding Gauss-Manin connection on the  $Sp(4,\C)$ Hitchin fibration. 
 
As in the case of $U(m,m)$-Higgs bundles of \cite{umm}, the invariant and anti-invariant sections of $L\in {\rm Prym}(S,\bar S)[2]$   decompose the direct image bundle $\rho_*U:=  U_+\oplus   U_-$, leading to the symplectic decomposition $E=W\oplus W^*$ as
$W:=\bar \pi_*U_+$ and $W^*:=\bar \pi_*U_-.$

 \begin{lemma}\label{lemmaU}
 Let $D\in \Z_2[a_m]/b_0$ be the divisor of degree $M$ on which the involution $\sigma$ acts as $-1$. Then, there exists a line bundle $L_0\in {\rm Prym}(\bar S, \Sigma)$ such that 
  \begin{eqnarray}
   U_-= U_+ \otimes  \CO(D)\otimes K^*\otimes L_0 . 
 \end{eqnarray}
 \end{lemma}
\begin{proof}The symplectic structure of $E$ is obtained though relative duality (e.g. see \cite[Section 4]{classes}), and in particular it implies that 
$\bar \pi_*U_+ \cong (\bar \pi_* U_-)^*.$
Hence, as in \cite[Section 5]{umm}, the line bundles $U_\pm$ are not independent. Indeed, from \cite[Eq.~(15)]{umm} one has that
\begin{eqnarray}
{\rm Nm}(U_+) = -{\rm Nm}(U_{-}) + 2m(m - 1)K.
\end{eqnarray}
 The above can be also written in terms of the divisor $D\in \Z_2[a_m]/b_0$ of degree $M$ on which the involution $\sigma$ acts as $-1$, as
$ D = {\rm Nm}(U_+^*) + {\rm Nm}(U_-) + mK.
$ Therefore, viewing $D$ as a divisor on $\bar S$ (since it is a subset of the ramification divisor of the $m$-fold cover $\bar \pi:\bar S\rightarrow \Sigma$),  up to a line bundle $L_0\in {\rm Prym}(\bar S, \Sigma)$, the result follows. 
 \end{proof}
 
 From \cite[Eq. (9) \& (10)]{umm} one can write the degrees of the line bundles $U_{\pm}$ in terms of $M$. 
In particular, recalling that $U=L  \otimes K^{(2m-1)/2}$ one has that $
\deg(U)= m(2m-1)(g-1),\nonumber
$ and therefore the degrees of the invariant and anti-invariant line bundles on $\bar S$ can be expressed as
\begin{eqnarray}
\deg(U_+)&=&m(2m-1)(g-1)-\frac{M}{2},\label{U+}\\
\deg(U_-)&=& m(2m-3)(g-1)+\frac{M}{2}. \label{U-}
\end{eqnarray}
Equivalently,  from \cite[Eq. (9)]{umm} one can write the degree of the rank $m$ bundle $W$ as
\begin{eqnarray}
\deg(W)= \frac{\deg(U)}{2}-\frac{M}{2} -(2m^2-2m)(g-1) =-\frac{M}{2}  +m(g-1),\nonumber
\end{eqnarray}
recovering the result in \cite[Eq.~(7)]{N2}.
The case of $M=0$ corresponds to the maximal Toledo invariant setting, for which it is known that within the covering space  one has $2^{2g}$ connected components, the so called {\it Hitchin components},   parametrising rich geometric structures. Moreover, when $m=2$ the $\Z_2$ vector space $H^1(\bar S, \Z_2)$ can be understood in terms of line bundles of order two over the Riemann surface $\Sigma$, and we shall comment on this case in Section \ref{last}.

\section{$SO(m,m+1)$-Higgs bundles}\label{section-so}

An  $SO(m,m+1)$ Higgs bundle is a pair $(V,\Phi)$ where $V=V_{+}\oplus V_{-}$ for $V_{\pm}$ complex vector spaces with orthogonal structures of dimension $m$ and $m+1$ respectively. The Higgs field is a section in $H^{0}(\Sigma, ({\rm Hom}(V_{-}, V_{+})\oplus {\rm Hom}(V_{+},V_{-}))\otimes K)$  given by
\[\Phi=\left(\begin{array}{cc}
              0&\beta\\
\gamma&0
             \end{array}
\right)~{~\rm for~}~\gamma \equiv -\beta^{\rm T},\] 
where $\beta^{\rm T}$ is the orthogonal transpose of $\beta$. As mentioned previously, since $SO(m,m+1)$ retracts onto $S(O(m)\times O(m+1))$, the Higgs bundles $(V_{+}\oplus V_{-}, \Phi)$ carry three topological invariants, the Stiefel-Whitney classes $\omega_1(V_+)$, and $\omega_2(V_\pm)$ (note that $\omega_1(V_-)=\omega_1(V_+)$ since $\det V_- =\det V_+^*$). By further requiring the Higgs bundle to be in the connected component of the identity, i.e. taking $SO(m,m+1)_0$-Higgs bundles, one would obtain pairs with $\omega_1(V_{\pm})=\CO$ (as considered, for instance, in \cite{aparicio}). 
In what follows we shall give a geometric description of these topological invariants, relate them to the ones for $Sp(2m,\R)$-Higgs bundles obtained in \cite[Section 6]{classes}, and finally use this description to characterise Zariski dense open sets in each connected component of the moduli space of $SO(m,m+1)$-Higgs bundles. On several occasions, it will be important to distinguish when $m$ is even or odd, and we shall do so within this section. 

\subsection{KO-theory of $\Sigma$}
In order to 
  discuss the topology of orthogonal bundles on the surface $\Sigma$ we use KO-theory. For this,  we shall recall some results from   \cite[Section 6]{classes} and    \cite{three}.
The Stiefel-Whitney classes  of $V_\pm$ can be seen as   classes $[V_{\pm}]\in KO(\Sigma)$ where 
\begin{eqnarray}
[V_{\pm}]\in KO(\Sigma)&\simeq& \Z \oplus H^{1}(\Sigma,\Z_2)\oplus \Z_2\nonumber\\
V_{\pm}&\mapsto&(rk(V_{\pm}), \omega_1(V_{\pm}), \omega_2(V_{\pm})).\nonumber
\end{eqnarray}

 Taking the map  given by the  total Stiefel-Whitney class $\omega = 1 + \omega_1  + \omega_2 $  to the multiplicative group $\Z \oplus H^{1}(\Sigma,\Z_2)\oplus \Z_2$, we consider the generators given by holomorphic line bundles $L$ such that $L^2 \simeq O$, and the class $\Omega = \CO_p + \CO_p^* - 2$ where $\CO_p$ is the holomorphic line bundle given by a point $p \in \Sigma$. Then, for $\alpha(x)$  the class of a line bundle $x\in H^1(\Sigma,\Z_2)$ and $(x,y)$ the intersection form,  
 $\alpha(x+y)=\alpha(x)+\alpha(y)-1+(x,y)\Omega.$
 As in \cite[Section 5]{classes}, the isomorphism between the additive group $\widetilde{KO}(\Sigma)$ and the multiplicative group $KO(\Sigma)$ is determined by the relations
 \begin{eqnarray}
  \omega_1(\alpha(x))=x~,~    \omega_1(\Omega)=0,
~{~\rm ~and~}~ \omega_2(\Omega)=c_1(\CO_p)~({\rm mod}~2)~=[\Sigma]\in H^{2}(\Sigma,\Z_2). \nonumber
  \end{eqnarray}
 With this notation, the classes $[V_{\pm}]$ satisfy 
$  [V_\pm ] = rk(V_\pm) - 1 + \alpha(\omega_1(V_\pm )) + \omega_2(V_{\pm})\Omega.\nonumber
$ 
Choosing a $\Theta$ characteristic $K^{1/2}$, the classes  $[V_{\pm}]$ have associated an analytic mod 2 index
\begin{eqnarray}\label{hola}
\varphi_{\Sigma}(V_{\pm} ) = \dim H^0(\Sigma, V_\pm \otimes K^{1/2}) ~({\rm mod} ~2),
\end{eqnarray}
and the characteristic class $\omega_2$ is independent of which spin structure $K^{1/2}$ is chosen. It follows  from \cite[Theorem 1]{classes} that the classes $\omega_2(V_{\pm}) $ satisfy
\begin{eqnarray}\omega_2(V_{\pm})= \varphi_{\Sigma}(V_{\pm})+\varphi_{\Sigma}(\det(V_{\pm})).\label{algo}\end{eqnarray}

 Moreover, $\varphi_{\Sigma}(\Omega)=1$ and the map  can be seen as the map to a point 
\[\varphi_{\Sigma}:KO(\Sigma)\rightarrow KO^{-2}(pt)\cong \Z_{2}.\]
 Since we are interested in understanding Higgs bundles through their spectral data, we note that as in \cite[Section 5]{classes}, the spin structures together with the covers $\pi:S\rightarrow \Sigma$ and $\bar \pi:\bar S\rightarrow \Sigma$ define push forward maps $KO(S)\rightarrow KO(\Sigma)$ and $
 KO(\bar S)\rightarrow KO(\Sigma)$.

 \subsection{Spectral data  for $SO(m,m+1)$-Higgs bundles} \label{some}
  In order to give a geometric description of  characteristic classes, 
  we shall define here  the spectral data associated to the  $SO(m,m+1)$-Higgs bundles. 
 One should note that since $SO(m,m+1)$-Higgs bundles lie completely inside the singular fibres of the    $SL(2m+1,\C)$ Hitchin fibration,   the analysis done in \cite[Section 5]{classes} can not be  directly applied.

In this case we consider points of order two in the fibres \eqref{fibreso} of the $SO(2m+1,\C)$-Hitchin fibration, which form two copies of the space 
\begin{eqnarray}
{\rm Prym}(S,\bar S)[2]/\rho^*H^1(\bar S, \Z_2),\label{quotient}
\end{eqnarray}
 a $\Z_2$ vector space of dimension $4m(g-1)-2.$  Moreover, the points in $\rho^*H^1(\bar S, \Z_2)$ are precisely those line bundles in ${\rm Prym}(S,\bar S)[2]$ with trivial action of $\sigma$ at all fixed points, i.e., with invariant $M=0$. From \cite[Theorem 4.12]{thesis} together with the structure of $\rho^*H^1(\bar S, \Z_2)$ from Lemma \ref{lemmaH},  the fibres \eqref{quotient} can be further described as follows.  
 
 \begin{proposition}\label{fibreSO}
 The intersection of the moduli space $\CM_{SO(m,m+1)}$ with the regular fibres of the $SO(2m+1,\C)$-Hitchin fibration is given by two copies of  
${\rm Prym}(S,\bar S)[2]~/~\rho^*H^1(\bar S, \Z_2)\nonumber
$ where the $\Z_2$ space $H^1(\bar S, \Z_2)
$  is given by
\begin{eqnarray}  
  {\rm Prym}(\bar S,\Sigma)[2]\oplus  H^1(\Sigma, \Z_2)~&{\rm for~}& m~{\rm odd},\label{seq-3} \\
 H^1(\Sigma, \Z_2)\oplus ({\rm Prym}(\bar S,\Sigma)[2]/  \pi^*H^1(\Sigma, \Z_2))\oplus   H^1(\Sigma, \Z_2)~&{\rm for~}& m~{\rm even}.\label{seq-4}
 \end{eqnarray}
 \end{proposition}
%

%

From Proposition \ref{fibreSO},  the spectral data associated to an $SO(m,m+1)$-Higgs bundle, up to equivalences by \eqref{seq-3}-\eqref{seq-4},  is given by the intermediate spectral curve $\bar S$ together with
 a line bundle $\CF\in H^1(\bar S, \Z_2)$, and
  a divisor $D\in \Z_2([a_m])^{ev}/b_0$ of degree $M$.

\begin{remark}It is interesting to note that when $m=2$ the middle term in \eqref{seq-4} gives in fact the spectral data for a $K^2$-twisted $PGL(2,\R)$-Higgs bundle. Moreover, the component $ {\rm Prym}(\bar S,\Sigma)[2]$ gives the spectral data for $K^2$-twisted $SL(m,\R)$-Higgs bundles. 
\end{remark}

In order to recover $SO(m,m+1)$-Higgs bundles from the above spectral data,  we shall recall the relation between symplectic and orthogonal Higgs bundles as described in \cite[Section 4]{N3}. 
 As mentioned in Section \ref{section-sp}, the line bundle $L\in {\rm Prym}(S,\bar S)[2]$ defines a symplectic vector bundle as $E:=\pi_*U$, for    $U=L\otimes K^{1/2}_S\otimes \pi^* K^{-1/2}.\nonumber 
$  Then,  from \cite[Eq.~(7)]{N3} the orthogonal bundle $V$ is recovered as an extension
 \begin{eqnarray}
 0\rightarrow E\otimes K^{-1/2}\rightarrow V\rightarrow K^{m}\rightarrow 0,\label{extO}
 \end{eqnarray}
and therefore near the divisor defined by the section $a_m$, the orthogonal bundle $V$ of the $SO(m,m+1)$-Higgs pair $(V,\Phi)$ is recovered as
 $ V:=(E\otimes K^{-1/2})\oplus K^m.$

From \cite[Section 4.1]{N3}, the $2m+1$ vector bundle $V$ has trivial determinant and a nondegenerate symmetric bilinear form $g(v,w)$ for which $g(\Phi v,w)+g(v,\Phi w)=0$, related to the symplectic form on $E$. Indeed, by considering the Higgs field $\Phi$ on $V/K^{-m}$, one has a non-degenerate skew form on $V/K^{-m}$, and by choosing a square root $K^{1/2}$,  one obtains a skew form on $E=V/K^{-m}\otimes K^{-1/2}$ which is generically non degenerate:
$\omega(v,w)=g(\Phi v, w).$
 Moreover, the extension class in \eqref{extO} can be seen as a choice of trivialization of the line bundle $L\in {\rm Prym}(S,\bar S)$  which depends on the action of the involution $\sigma$, this is, on the divisor \linebreak$D\in \Z_2[a_m]/b_0$ (see \cite[Section 4.3]{N3}).
 
 The orthogonal structure induced on the rank $m$ and $m+1$ vector bundles 
$V_+\oplus V_-$ 
obtained through the spectral data 
in the fibre \eqref{quotient} can be understood in terms of a decomposition of the symplectic bundle 
$E:=E_+\oplus E_-$, through which locally one has
 \begin{eqnarray}
 V_-&=&E_-\otimes K^{-1/2}\oplus K^m,\label{bundlev-}\\
V_+&=& E_+\otimes K^{-1/2}.\label{bundlev+}
\end{eqnarray}

One should note that it is not the symplectic decomposition $E=W\oplus W^*$ which leads to the decomposition $E=E_+\oplus E_-$ on the orthogonal side. This becomes evident, for instance,  by considering the Hitchin components for both groups, and it is described in Section \ref{hitchin}.
%
   Furthermore, since $V_\pm$ form part of $GL(m,\C)$ and $GL(m+1,\C)$ Higgs bundles, from  \cite{N2} and \cite{bnr} there is a line bundle  on $\bar S$ whose direct image gives $V_+$ on $\Sigma$. 
Adopting   the notation of \cite[Section 5]{classes} we define $\pi_!$ and $\bar \pi_!$ by
\begin{eqnarray}
\pi_!(\CL)=\pi_*(\CL \otimes K^{1/2}_S \otimes \pi^*K^{-1/2}),~{\rm and ~} 
\bar\pi_!(\CL)=\pi_*(\bar\CL \otimes K^{1/2}_{\bar S} \otimes \pi^*K^{-1/2}),\label{orthogonal}
\end{eqnarray}
for $\CL$ and $\bar \CL$ line bundles on $S$ and $\bar S$. Then, as seen in Section \ref{section-sp}, the symplectic vector bundle $E$ is obtained as
$E:= \pi_!(L)$ 
for $L\in {\rm Prym}(S,\bar S)$. 
 When $\CL^2\cong \CO$ and $\bar L^2\cong \CO $,  the  bundles $\pi_!(\CL)$ and $\bar \pi_!(\bar \CL)$ acquire   orthogonal structures by relative duality, as shown in \cite[Section 4]{classes}. Hence, since $V_+$ has an orthogonal structure, following \cite[Section 4]{classes} and \cite{bnr} for $K^2$-twisted Higgs bundles, the vector bundle $V_+$ is obtained,     for some     $\CF\in H^1(\bar S, \Z_2), $ as 
    \begin{eqnarray}
  V_+= \bar \pi_!(\CF).
  \end{eqnarray}
  
 \begin{lemma}\label{lemmaF}
 For   $\CF\in H^1(\bar S, \Z_2)$, one has  $\det (\bar \pi_!(\CF))={\rm Nm}( \CF)$.
 \end{lemma}
\begin{proof}  The determinant bundle of $\bar \pi_!(\CF)$ can be obtained through \cite[Section 4]{bnr}, leading to 
$\det (\bar \pi_!(\CF))=  {\rm Nm}(\CF \otimes K^{1/2}_{\bar S} \otimes \bar \pi^*K^{-(2m-1)/2} )
  =  {\rm Nm}( \CF).$ 
\end{proof}
%
%

   In order to understand how the other orthogonal bundle $V_-$ is reconstructed, we shall give now a construction of $E_-$ via the spectral data $\CF$ and $D$ modulo \eqref{seq-3}-\eqref{seq-4} (which in particular implies modulo ${\rm Prym}(\bar S,\Sigma)$).   Since $\det(E_+)\otimes \det (E_-)=\CO$, from \eqref{bundlev-}-\eqref{bundlev+} one has that $ \det(E_-)= {\rm Nm}(\CF)\otimes K^{-m/2}$. Therefore, for some $L_0\in {\rm Prym}(\bar S,\Sigma)$ one may write
   \begin{eqnarray}V_-= \bar \pi_*(L_0\otimes K_{\bar S}^{-1/2}\otimes \CF)\otimes K^{-1/2}. \end{eqnarray}
  Note that the choice of $L_0$ is equivalent to the one done  in Lemma \ref{lemmaU}, and th divisor $D$ gives the extension class as in the complex case described in \cite[Section 4.3]{N3}.

  \subsection{Characteristic classes for $SO(m,m+1)$-Higgs bundles}

  In what follows we shall see that the three Stiefel-Whitney classes of $SO(m,m+1)$-Higgs bundles $(V_+\oplus V_-, \Phi)$    can be described in terms of their spectral data, which from the previous sections is given modulo \eqref{seq-3}-\eqref{seq-4} by  
  \[(\CF,D)\in H^{1}(\bar S, \Z_2)\oplus \Z_2([a_m])/b_0.\]
  
   \begin{theorem}\label{teo2}
The  Stiefel-Whitney classes  of an $SO(m,m+1)$-Higgs bundle $(V_-\oplus V_+, \Phi)$ with spectral data $(\CF, D)\in {\rm Prym}(S,\bar S)[2]/ \rho^*H^1(\bar S, \Z_2)$ are given by
\begin{eqnarray}
\omega_1(V_+)&=&  {\rm Nm}(\CF)\in H^1(\Sigma, \Z_2),\\
\omega_2(V_+)&=&\varphi_{\bar S}(\CF) +\varphi_\Sigma({\rm Nm}(\CF))   \in \Z_2,\\
\omega_2(V_-)&=&\left\{\begin{array} {ccc}
\omega_2(V_+)&{\rm if} &\omega_2(V)=0,\\
\omega_2(V_+)+1&{\rm if} &\omega_2(V)=1.
\end{array}\right.\end{eqnarray}
 \end{theorem}

 \begin{proof}
 
 Recall that $
\varphi_{\Sigma}(\CL) = \dim H^0(\Sigma, \CL \otimes K^{1/2}) ~({\rm mod} ~2),$ and from \cite[Theorem 1]{classes} that for an even spin structure $K^{1/2}$, the orthogonal bundles $V_{\pm}$ satisfy
\begin{eqnarray}\omega_2(V_{\pm})= \varphi_{\Sigma}(V_{\pm})+\varphi_{\Sigma}(\det(V_{\pm})).\label{w2}\end{eqnarray}
Moreover, since $\deg(\det(V_{\pm}))=0$, one has that $\varphi_\Sigma(V_-)=\varphi_\Sigma(V_+)~({\rm mod}~2).$

The above  can also be seen in terms of the analytic mod 2 indices  $\varphi_S$ and $\varphi_{\bar S}$ of the spectral line bundles producing $V_+$ and $V_-$. The three mod 2 indices can be related by considering the definition of push forward of sheaves. Indeed, note that for $\CL$ a torsion two line bundle on $S$,   by definition of direct image sheaf 
\begin{eqnarray}
\varphi_S(\CL)=\dim H^0(S,  \CL \otimes K^{1/2}_S)~({\rm mod} ~2)
=\dim H^0(\Sigma,  \pi_*(\CL \otimes K^{1/2}_S \otimes \pi^*K^{-1/2}) \otimes K^{1/2})~({\rm mod} ~2),\nonumber
\end{eqnarray}
and hence
$\varphi_S(\CL)=\varphi_\Sigma(\pi_!(\CL)). $
An equivalent formula follows for $\bar S$, 
and therefore 
\begin{eqnarray}\omega_1(V_+)={\rm Nm}(\CF)\in H^1(\Sigma, \Z_2).\end{eqnarray}
Moreover,  since $\varphi_{\bar S}(\CF)=\varphi_\Sigma(\bar \pi_!(\CF))
$ 
it follows that 
\begin{eqnarray}
\omega_2(V_+)=\varphi_{\bar S}(\CF) +\varphi_\Sigma({\rm Nm}(\CF)).
\end{eqnarray}

In order to understand $\omega_2(V_-)$ through \eqref{w2}, we should recall that 
$\omega_2(V)=\omega_2(V_+)+\omega_2(V_-),$
and thus 
$ \omega_2(V_-)=\varphi_{\bar S}(\CF)+\varphi_{\Sigma}({\rm Nm}(\CF)) +\varphi_{\Sigma}(V).
$ The value of $\varphi_\Sigma(V)=\omega_2(V)$ has been studied in \cite[Section 4]{N3} and indicates whether $V$ has a lift to a spin bundle or not. In particular, it is shown there that it is the identity component of the fibre that gives spin bundles, which is  $\omega_2(V)=\varphi_\Sigma(V)=0$, and the theorem follows. 
 \end{proof}

  One should note that when further requiring $V_+$ to have trivial determinant, it becomes the vector bundle of an $SL(m,\R)$-Higgs pair and our result agrees with the description of $\omega_2(V_+)$ of  \cite[Theorem 1]{classes} for  a fixed even spin structure.
  When considering $SO(m,m+1)_0$-Higgs bundles, i.e. Higgs bundles in the component of the identity, both vector bundles $V_\pm$ satisfy $\det V_{\pm}=\CO$, and thus they are obtained by choosing a point $L_+\otimes \bar \pi^{*}K^{3/2}$ in the Prym variety ${\rm Prym}(\bar S, \Sigma)$, after fixing a choice of spin structure $K^{1/2}$. Moreover, in this case $M=4m(g-1)$ and thus $L^* \in {\rm Prym}(S,\Sigma)[2]$ is the pullback of a line bundle on $\bar S$, hence determined by $L_+$. 
 Since the characteristic class $\omega_2$ is independent of which spin structure $K^{1/2}$ is chosen, we may use this fact to further deduce the following from Theorem \ref{teo2} by fixing $K^{1/2}$ such that $\varphi_{\Sigma}(\CO)=1$,  along the lines of \cite[Theorem 1]{classes} purely in terms of spin structures: 
 \begin{corollary}\label{corro}
 Let $\bar S$ be a smooth spectral curve in the total space of $K^2\rightarrow \Sigma$ give by an equation
 \[ \eta^m+a_1\eta^{m-1}+\ldots+a_{m-1}\eta+a_m=0,\]
 and let $\CF$ be a line bundle on $\bar S$ such that $\CF^2\cong \CO$. Define $V_+:=\pi_!(\CF)$ the image bundle with the orthogonal structure induced from relative duality. 
 Let $K^{1/2}$ be an even spin structure on $\Sigma$,   and for  $K_{\bar S}^{1/2}=\bar \pi^*K^{m-1/2}$ the corresponding one on $\bar S$, consider the spin structure  $\CF_{\bar S}=\CF \otimes K_{\bar S}^{1/2}$. Then, the characteristic classes of the corresponding $SO(m,m+1)$-Higgs pair are
\begin{eqnarray}
\omega_1(V_+)&=&  {\rm Nm}(\CF)\in H^1(\Sigma, \Z_2),\\
\omega_2(V_+)&=&\left\{\begin{array} {ccc}
\varphi_{\Sigma}({\rm Nm}(\CF))&{\rm if} &\varphi_S(\CF_{\bar S})=0 \\
1+\varphi_{\Sigma}({\rm Nm}(\CF))&{\rm if} &\varphi_S(\CF_{\bar S})=1
\end{array}\right.,\\
\omega_2(V_-)&=&\left\{\begin{array} {ccc}
\varphi_{\Sigma}({\rm Nm}(\CF))&{\rm if} &\varphi_S(\CF_{\bar S})=\varphi_\Sigma(V)\\
1+\varphi_{\Sigma}({\rm Nm}(\CF))&{\rm if} & \varphi_S(\CF_{\bar S}) \neq\varphi_\Sigma(V)\end{array}\right..\end{eqnarray}
   \end{corollary}
  
\subsection{The divisor $D\in \Z([a_m])/b_0$}
  We shall finally consider the geometric implications of the divisor $D\in \Z([a_m])/b_0$ appearing in the spectral data of the Higgs bundles studied in this paper. As mentioned previously,   the extension class giving the orthogonal bundle $V$ is obtained through $D$. Moreover, its degree $M$ appears both at the level of complex $SO(2m+1,\C)$-Higgs bundles (see \cite[Remark 2 p.14]{N3}) and real $SO(m,m+1)$-Higgs bundles.   
  
From \cite[Section 6]{classes}, and as recalled in Section \ref{section-sp},   the elements in ${\rm Prym}(S,\bar S)[2]$ can be distinguished by their associated invariant $M$, and in each regular fibre there are \begin{small}
\begin{eqnarray}
\left(\begin{array}{c} 4m(g-1)\\M
\end{array}\right)\times 2^{2g_{\bar S}}\nonumber
\end{eqnarray}\end{small}
 points with invariant $M$, where the genus of $\bar S$ is as before, $g_{\bar S}=( 2 m^2 - m ) ( g - 1 ) + 1$. Hence, in order to differentiate the characteristic classes for   $SO(m,m+1)$-Higgs bundles, one needs to understand the characteristic classes of $\rho^*H^1(\bar S, \Z_2)$. In particular, it should be noted that since the Prym variety ${\rm Prym}(S,\bar S)$ is defined as the set of line bundles $L\in {\rm Jac}(S)$ such that $\sigma^*L\cong L^*$, the pulled-back line bundles in  $\rho^*H^1(\bar S, \Z_2)$ are acted on trivially by the involution $\sigma$ and thus carry invariant $M=0$. Therefore, recalling that the topological invariant $M$ associated to $SO(m,m+1)$-Higgs bundles can be seen from \eqref{am} as the degree of the subdivisor of $[a_m]$ giving an element in  $\Z_2([a_m])^{ev}$, one has the following:
 
 \begin{proposition}\label{numbers}
In each generic fibre there are  \begin{small}
$\left(\begin{array}{c} 4m(g-1)\\M
\end{array}\right)
$ \end{small} points  with even invariant $M$.\end{proposition}
  Since exchanging $\sigma$ by $-\sigma$ exchanges the values of $M$ and $4m(g-1)-M$, those two cases should be identified. Hence,    the total number of points in each regular fibre is half of 
 \begin{eqnarray}
 \left(\begin{array}{c} 4m(g-1)\\0
\end{array}\right)+ \left(\begin{array}{c} 4m(g-1)\\2
\end{array}\right)+\ldots+ \left(\begin{array}{c} 4m(g-1)\\4m(g-1)-2
\end{array}\right)+ \left(\begin{array}{c} 4m(g-1)\\4m(g-1)
\end{array}\right),
 \end{eqnarray}
 which using series multisection gives, as expected, 
$  \frac{1}{2}\left[ 2^{4m(g-1)-1} \right]. $
 

 \subsection{On the geometry of the moduli space}\label{geometrytori}

From the above analysis, one has a natural grading of the moduli space leading to a  geometric description of Zariski dense open sets in the moduli space of $SO(m,m+1)$-Higgs bundles:
\begin{theorem}\label{teo1}
For each fixed even invariant $0<M\leq 4m(g-1)$, the moduli space of \linebreak $SO(m,m+1)$-Higgs bundles intersects the regular fibres of the Hitchin fibration in a   component given by a covering of a vector space over the symmetric product $S^M\Sigma$.  \end{theorem}
\begin{proof}
 Over a point in the Hitchin base, defining a spectral curve, one has  $\Z_2([a_m])^{ev}/d_0 $. This is all choices of $4m(g-1)$ $\Z_2$-uples $D$ with an even number of $+1$, up to the element $(1,\ldots, 1)$ and equivalence. As in \cite[Proposition 4]{classes}, this agrees with the previous section, asserting that the intersection of the space with the fibre is a $\Z_2$ vector space of dimension
    $ 4m(g-1)-2$.
     
    As in \cite[Theorem 4.2]{umm}, the choice of the divisor $D$ (or equivalently, a point in $\Z_2([a_m])^{ev}/d_0$)  is given by a point in the symmetric product $S^M\Sigma$, for $M$ the degree of $D$.
     Then, the choice of $a_m$ is given by the choice of a section $s\in H^0(\Sigma, K^{2m}(-D))$, leading to a vector bundle $B$ of rank $(4m-1)(g-1)- M$ over $S^M\Sigma$. 
     Finally, the choice of the spectral curve is completed by considering, as in the symplectic side, the space 
$\bigoplus_{i=1}^{2m-2}H^0(\Sigma, K^{2i}), $
 where the parametrisation is done up to $H^1(\bar S, \Z_2)$.
\end{proof}
 One should keep in mind that the characteristic classes of the $SO(m,m+1)$-Higgs bundles are topological invariants, and thus are constant within connected components. On the other hand, the invariant $M$ labels components which often intersect over the singular locus of the Hitchin fibration. An interesting comparison can be made with \cite[Section 6.4]{brian}, where it is shown how the invariant $M$ labels certain connected components of the moduli space. One should note also that the space $H^1(\bar S, \Z_2)$ is in fact the spectral data for $K^2$-twisted $GL(m,\R)$-Higgs bundles, and thus over each point in the Hitchin base one has the fibre giving the spectral data for a corresponding $K^2$-twisted $GL(m,\R)$-Higgs bundle, the Cayley partners.
 \begin{corollary}\label{cor1}
When $M=0$ and $m$ is odd the intersection with smooth fibres is given by $2^{2g}$ copies of ${\rm Prym}(\bar S,\Sigma)$ over a vector space. 
\end{corollary}

\section{Concluding remarks}\label{last}

In what follows, we shall describe some applications of the above methods in the context of understanding the moduli spaces for other real groups. 

\subsection{The $Sp(2m,\R)$ and $SO(m,m+1)$ Hitchin components} \label{hitchin}When considering the Hitchin components for both split real forms $Sp(2m,\R)$ and $SO(2m+1,\R)$ as described in \cite{N5}, one can see that the vector bundle for $Sp(2m,\R)$ is given by (e.g. see \cite[p.4]{classes})
 \begin{eqnarray}
  E:=\bigoplus_{i=1}^{2m} \left (K^{-m+i}\otimes K^{-1/2} \right ).
  \end{eqnarray}
   Then, by considering $E\otimes K^{-1/2}\oplus K^m$ one obtains, as expected, the  orthogonal  bundle  for $SO(m,m+1)$-Higgs bundles 
    \begin{eqnarray}
  V:= \bigoplus_{i=0}^{2m}K^{-m+i}.    \end{eqnarray}
  The pairing for the symplectic vector bundle $E=W\oplus W^*$ is obtained by considering the symplectic pairing between $K^{\pm a}$, leading to
  \begin{eqnarray}
  W= \bigoplus_{i=m+1}^{2m} \left (K^{-m+i}\otimes K^{-1/2} \right ).
  \end{eqnarray}
  On the other hand, the pairing for the orthogonal bundle $V=V_+ \oplus V_-$ is obtained by taking the natural orthogonal structure for each $K^a\oplus K^{-a}$ and thus one has for $m$ even
  \begin{eqnarray}
  V_-=\bigoplus_{i=0}^{m-1} K^{-m+2i+1};~{~\rm~and ~}
  V_+=\bigoplus_{i=0}^{m}  K^{-m+2i} ~= ~K^m \oplus \bigoplus_{i=0}^{m-1}  K^{-m+2i}, \label{v-}
  \end{eqnarray}
  and if $m$ is odd, the roles of $V_-$ and $V_+$ are interchanged. 
 %
One should note that, in particular, separating the vector bundle  $W=W_{+}\oplus W_{-}$ into the odd (-) and even (+) values of $i$, one has  
\begin{eqnarray}V_-&=&W_{-}\oplus W^*_{-} ~{\rm ~ and ~}~
V_+=W_{+}\oplus W^*_{+},
\end{eqnarray}
and thus the relation between both decompositions of the symplectic bundle and orthogonal bundle become apparent. 
In the case of $SL(m,\R)$-Higgs bundles, the Hitchin component is given by Higgs bundles whose underlying vector $\tilde \CV$ bundle has the form (see \cite[Section 3]{N5}) 
\begin{eqnarray} 
\tilde \CV=\bigoplus_{i=0}^{m-1} K^{\frac{-m+1+2i}{2}}.
\end{eqnarray}
This bundle is obtained from the origin in the fibre of the Hitchin fibration, and a similar construction leads to the Hitchin component for $K^2$ twisted $SL(m,\R)$-Higgs bundles, where
$\CV=\bigoplus_{i=0}^{m-1} K^{-m+1+2i}.
$ In particular, this rank $m$ vector bundle coincides with $V_-$ in  \eqref{v-}, which is not surprising, as it follows from the proof of Theorem \ref{teo2}. 

\subsection{Maximal $Sp(4,\R)$-Higgs bundles and $SO(3,4)$-Higgs bundles}\label{mono}

It is interesting to consider the description of the connected components given in Section \ref{section-sp} for $m=2$, or in other words, for $Sp(4,\R)$-Higgs bundles. In this case, it was shown in Gothen's thesis (see \cite[p. 831-849]{gothen}) that for maximal Toledo invariant (i.e. $M=0$), the number of connected components is 
$3\cdot 2^{2g}+2g-4.$ 
As in the general case,  the components are described by $H^1(\bar S, \Z_2)$ over a vector bundle over the symmetric product $S^M\Sigma$. But in the case of $m=2$ one has one more correspondence to consider. Indeed, $H^1(\bar S, \Z_2)$ becomes the spectral data for $K^2$-twisted $GL(2,\R)$-Higgs bundles, the Cayley partner of $Sp(4,\R)$-Higgs bundles. As in \cite[Theorem 6.8]{mono}, one has that 
$H^1(\bar S, \Z_2)=\Lambda_{\Sigma}[2]\oplus \Z_2([a_m])^{ev}/b_0 \oplus \Lambda_{\Sigma}[2]
$ and as seen in \cite[Corollary 6.9]{mono}, one recovers the  $3\cdot 2^{2g}+2g-4$ components as orbits of the monodromy action. Moreover, from the description in Section \ref{section-sp}, these components appear as the components of $K^2$-twisted Higgs bundles over the vector space $\mathcal{A}$. The geometry of these components can be studied by similar methods to those in Section \ref{section-sp} and Section \ref{section-so}, by noting that a choice in $\Z_2([a_m])^{ev}/b_0$ gives a point in a symmetric product labeled by the invariant $M$, and over that one has $2^{2g}$ covers coming from $H^1(\Sigma,\Z_2)$.
%

\subsection{The dual $(B,B,B)$-branes} The smooth locus of the moduli space of $SO(2m+1,\C)$-Higgs bundles on  $\Sigma$  is a hyper-K\"ahler manifold, so there are natural complex structures $I,J,K$ obeying the same relations as the imaginary quaternions (following the notation of \cite{LPS_Kap}).
  Adopting physicists' language,  a Lagrangian submanifold of a symplectic manifold is called an {\em A-brane} and a complex submanifold a {\em B-brane}. A submanifold of a hyper-K\"ahler manifold may be of type $A$ or $B$ with respect to each of the complex or symplectic structures, and thus choosing a triple of structures one may speak of branes of type $(B,B,B), (B,A,A), (A,B,A)$ and $(A,A,B)$.  
 The moduli space of $SO(m,m+1)$-Higgs bundles is a $(B,A,A)$-brane in the moduli space $\CM_{SO(2m+1,\C)}$ of complex Higgs bundles. As such, it has a dual $(B,B,B)$-brane in the dual moduli space $\CM_{Sp(2m,\C)}$ (see \cite{LPS_Kap}).
  It was conjectured  by D.~Baraglia and the author, in \cite[Section 7]{slice}, that the support of the dual $(B,B,B)$-brane should be the whole moduli space of $Sp(2m,\C)$-Higgs bundles, which can now be understood through the spectral data description of the components of $\CM_{SO(m,m+1)}$ given in this paper, and compared to the hyperholomorphic $(B,B,B)$-brane constructed by Hitchin in \cite[Section 7]{classes} dual to the $U(m,m)$-Higgs bundles  studied in \cite{umm}. In particular, $SO(m,m+1)$ and $U(m,m)$-Higgs bundles provide examples of $(B,A,A)$-branes whose dual $(B,B,B)$-brane should have the same support, and the moduli spaces of split Higgs bundles in Section \ref{hitchin} provide examples of associated $(B,A,A)$-branes.

\end{document}